\documentclass[letterpaper, 10 pt, conference]{ieeeconf}  

\IEEEoverridecommandlockouts                              

\usepackage{amsmath}
\usepackage{epstopdf}
\usepackage{color}

\usepackage{amsmath} 
\usepackage{amssymb}  

\newtheorem{example}{Example}[section]

\newtheorem{lemma}{Lemma}[section]
\newtheorem{theorem}{Theorem}[section]
\newtheorem{proposition}{Proposition}[section]
\newtheorem{definition}{Definition}[section]

\title{\LARGE \bf
	Observability for port-Hamiltonian systems}

\author{Birgit Jacob$^{1}$, Hans Zwart$^{2}$
   \thanks{$^{1}$Birgit Jacob is with the School of Mathematics and Natural Sciences, IMACM, University of Wuppertal, Gau\ss stra\ss e 20, D-42119 Wuppertal, Germany. {\tt\small bjacob@uni-wuppertal.de}}%
      \thanks{$^{2}$Hans Zwart is with the Department of Mechanical Engineering,  Dynamics and Control group,
    	Eindhoven University of Technology, 5612 AZ Eindhoven, The Netherlands
    	{\tt\small h.j.zwart@tue.nl}, and with the Faculty of Electrical Engineering, Mathematics and Computer Science, Department of Applied Mathematics, University of Twente, 7500 AE Enschede, The Netherlands
    	{\tt\small h.j.zwart@utwente.nl}}%
}

\begin{document}

\maketitle
\thispagestyle{empty}
\pagestyle{empty}

\begin{abstract}

The class of port-Hamiltonian systems incorporates many physical models, such as mechanical systems in the finite-dimensional case and wave and beam equations in the infinite-dimensional case. In this paper we study a subclass of linear first order port-Hamiltonian systems. In \cite{JacKai19}, it is shown that these systems are exactly observable when the energy is not dissipated internally and when sufficient observations are made at the boundary. In this article we study the observability properties for these systems when internal dissipation of energy is possible. We cannot show the exact observability, but we do show that the Hautus test is satisfied. In general, the Hautus test is weaker than exact observability, but stronger than approximate observability.  Hence we conclude that these systems are approximately observable. 
\end{abstract}


\section{Introduction}

In this paper we investigate the exact observability of port-Hamiltonian systems of the form 
\begin{align*}
\frac{\partial x}{\partial t}(\zeta,t) &= \left( P_1
  \frac{\partial}{\partial \zeta} + G_0\right) ({\mathcal H}({\zeta})
x(\zeta,t)),  \\
x(\zeta,0) &= x_0(\zeta),\\
0&= \begin{bmatrix} W_{B,1} & W_{B,0} \end{bmatrix} \begin{bmatrix} ({\mathcal H}x)(1,t) \\ ({\mathcal
         H}x)(0,t)\end{bmatrix}, \\
       y(t) &= \begin{bmatrix} W_{C,1} & W_{C,0} \end{bmatrix} \begin{bmatrix} ({\mathcal H}x)(1,t) \\ ({\mathcal
         H}x)(0,t)\end{bmatrix}, 
\end{align*}
where $\zeta \in [0,1]$ and $t\ge 0$.
We refer to Section \ref{sec2} for the precise assumptions on $P_1$, $G_0$, $\begin{bmatrix} W_{B,1} & W_{B,0} \end{bmatrix}$, $\begin{bmatrix} W_{C,1} & W_{C,0} \end{bmatrix}$ and ${\mathcal H}$.

This class of systems describes a special class of port-Hamiltonian systems, which however is rich enough 
to cover in particular the wave equation, the transport equation and the Timoshenko beam 
equation, and also coupled beam and wave equations.
For more information on this class of port-Hamiltonian systems we refer to the first paper on this topic, \cite{vdSM02}, the monograph \cite{JacZwa12} and the survey \cite{JacZwa19}.

Provided the port-Hamiltonian system is well-posed,  that is, for every  initial condition $x_0$ the port-Hamiltonian system possesses a unique (mild) solution, we aim  to characterize {\em observability} and to study the related Hautus test.
On one hand, exact observability implies the Hautus test and on the other hand, the Hautus test implies approximate observability, see Section \ref{backg} for the definitions of these notions.

In \cite{JacKai19} it is shown, that if additionally to the general assumptions of Section \ref{sec2} the matrix $G_0$ is skew-adjoint and the norm of the solution $x$ is non-increasing in time, then the  port-Hamiltonian system is exactly observable. Actually, in  \cite{JacKai19} the dual statement concerning controllability is shown. In this article, we provide a simple proof if the port-Hamiltonian system is strongly stable and $G_0$ is skew-adjoint. Moreover, we show that the Hautus test is satisfied for exponentially stable port-Hamiltonian systems. Thus these systems are in particular approximately observable.

\section{Background on exact observability}\label{backg}

In this section we assume that $A$ is the generator of a $C_0$-semigroup $(T(t))_{t\ge 0}$ on a Hilbert space $X$ and $C$ is a linear bounded operator
from $\mathbf{D}(A)$, equipped with the graph norm, to another Hilbert space $Y$. We now consider the observation system 
\begin{align}
  \dot{x}(t) &= Ax(t), \quad t\ge 0, \qquad x(0)=x_0,\label{eqnobs1}\\
  y(t) &= Cx(t), \quad t\ge 0.\label{eqnobs2}
\end{align}
If the initial condition $x_0$ is an element of $\mathbf{D}(A)$, then the unique (classical) solution of \eqref{eqnobs1}-\eqref{eqnobs2} is given by
$$ x(t)=T(t)x_0,\qquad y(t) =CT(t)x_0, \qquad t\ge 0,$$
thanks to the fact that $\mathbf{D}(A)$ is a $(T(t))_{t\ge 0}$-invariant subspace of $X$. In order to define a (mild) solution of \eqref{eqnobs1}-\eqref{eqnobs2} for arbitrary initial conditions $x_0\in X$ an extra condition is required.

\begin{definition}
The operator $C$ is called an \emph{admissible observation operator for $(T(t))_{t\ge 0}$}, if there is a time $t_0>0$ and a constant $m>0$ such that
\begin{align}\label{adm-1}
\int_0^{t_0} \|CT(t)x_0\|^2 dt\le m \|x_0\|^2, \quad \mbox{for all }   x_0\in \mathbf{D}(A).
\end{align}
\end{definition}

\vspace{1ex}
There is a rich literature on admissible observation operators, see for example \cite{japa04}. If $C$ is an admissible observation operator for $(T(t))_{t\ge 0}$, then \eqref{adm-1} holds for every $t_0>0$, where the constant $m$ may depend on $t_0$. If the constant $m$ can be chosen independent of $t_0$, then  $C$ is called an \emph{infinite-time admissible observation operator for $(T(t))_{t\ge 0}$}. If the $C_0$-semigroup $(T(t))_{t\ge 0}$ is exponentially stable, then the notion of admissibility and infinite-time admissibility are equivalent. Moreover, if $C$ is an admissible observation operator for $(T(t))_{t\ge 0}$ and $\alpha\in \mathbb R$, then $C$ is also an admissible observation operator for $(e^{\alpha t}T(t))_{t\ge 0}$.
Admissibility of the observation operator $C$ guarantees that the mapping
$$ \Psi_{t_0}: \mathbf{D}(A)\rightarrow L^2(0,t_0;Y),\qquad x_0\mapsto CT(.)x_0$$
has a continuos extension to a linear bounded operator from $X$ to $L^2(0,t_0;Y)$, again denoted by $\Psi_{t_0}$. Thus, in this case for every $x\in X$ the  unique (mild) solution of \eqref{eqnobs1}-\eqref{eqnobs2} is given by
$$ x(t)=T(t)x_0,\qquad y =\Psi_{t_0}x_0, \qquad t\in [0,t_0].$$
In this article we are in particular  interested in observability, which we define next.

\begin{definition}
The observation system \eqref{eqnobs1}-\eqref{eqnobs2} is called \emph{exact observable in finite time}, if there exists a time  $t_0>0$ and a constant $\delta>0$ such that
\begin{align}\label{obs}
\int_0^{t_0} \|CT(t)x_0\|^2 dt\ge \delta \|x_0\|^2, \quad \mbox{for all }   x_0\in \mathbf{D}(A),
\end{align}
 the observation system \eqref{eqnobs1}-\eqref{eqnobs2} is called \emph{exact observable in infinite time}, if there exists  a constant $\delta>0$ such that
\begin{align}
\label{ex-ob}
\int_0^{\infty} \|CT(t)x_0\|^2 dt\ge \delta \|x_0\|^2, \quad \mbox{for all }   x_0\in \mathbf{D}(A).
\end{align}
Further, the observation system \eqref{eqnobs1}-\eqref{eqnobs2} is called \emph{approximate observable in infinite time}, if 
\begin{align*}
\int_0^{\infty} \|CT(t)x_0\|^2 dt>0, \quad \mbox{for all }   x_0\in \mathbf{D}(A), x_0\not=0.
\end{align*}
\end{definition}

\vspace{1ex}
If $C$ is an admissible observation operator for $(T(t))_{t\ge 0}$, then the observation system \eqref{eqnobs1}-\eqref{eqnobs2} is exact observable in finite-time if and only if there is a $t_0> 0$ such that the mapping $\Psi_{t_0}$ has a linear bounded left-inverse. Further, if the $C_0$-semigroup $(T(t))_{t\ge 0}$ is exponentially stable, then the notion of observability in finite time and infinite time   are equivalent. We recall, that a $C_0$-semigroup $(T(t))_{t\ge 0}$ is \emph{exponentially stable}, if there exist constants $M\ge 1 $ and $\omega>0$ such that
\[ 
  \|T(t)\|\le M e^{-\omega t}, \qquad t\ge 0.
\]
A $C_0$-semigroup $(T(t))_{t\ge 0}$ on a Hilbert space $X$ is \emph{strongly  stable}, if 
\[
  \lim_{t\rightarrow \infty}\|T(t)x\|=0, \qquad x\in X.
\]
Next we provide an equivalent characterization of exact observability in term of Lyapunov (in)equalities. 
\begin{proposition}\label{Lyap0}
Let $A$ be the generator of a strongly stable $C_0$-semigroup $(T(t))_{t\ge 0}$ on a Hilbert space $X$. Then the following statements are equivalent
\begin{enumerate}
\item The observation system \eqref{eqnobs1}--\eqref{eqnobs2} is exact observable in infinite time.
\item There exists a linear bounded self-adjoint operator $L$ on $X$ satisfying
\begin{align}\label{coer}
 \langle x, Lx \rangle \ge \delta \|x\|^2
 \end{align}
for all $x\in X$ and some $\delta>0$, and 
\begin{align}\label{lyap-eq}
 \langle Ax, Lx \rangle + \langle Lx, Ax \rangle  = - \|Cx\|^2 
 \end{align}
for every $x\in \mathbf{D}(A)$.
\item There exists a linear bounded self-adjoint operator $L$ on $X$ satisfying (\ref{coer}) 
for all $x\in X$ and some $\delta>0$, and 
\begin{align}\label{lyap-ineq}
 \langle Ax, Lx \rangle + \langle Lx, Ax \rangle   \ge - \|Cx\|^2 
 \end{align}
for every $x\in \mathbf{D}(A)$.
\end{enumerate}
\end{proposition}
Inequality \eqref{lyap-ineq} is a Lyapunov inequality.

\begin{proof}
The implication $1) \Rightarrow 2)$ is well-know. For bounded operators $C\in L(X,Y)$ the proof can be found in  \cite[Exercise 6.10]{CuZw2020}. However, the general proof hardly differs. Furthermore, it is clear that $2) \Rightarrow 3)$. So we concentrate on the implication 3) $\Rightarrow$ 1).

For $x_0\in \mathbf{D}(A)$ we calculate 
 \begin{align*} 
 \frac{d}{dt} \langle T(t)x_0, L T(t)x_0\rangle = &\   \langle AT(t)x_0, L T(t)x_0\rangle + \\
 &\ \langle T(t)x_0, L AT(t)x_0\rangle\\
\ge  &\      -\|C T(t) x_0\|^2.
\end{align*}
In the last step we used (\ref{lyap-ineq}) with $x=T(t)x_0$.
Integrating both sides from $t=0$ to $t=t_0$, gives
\[
   \langle T(t_0)x_0, L T(t_0)x_0\rangle -  \langle x_0, L x_0\rangle \geq - \int_0^{t_0} \|C T(t) x_0\|^2 dt.
\]
Since the semigroup $(T(t))_{t\ge 0}$ is strongly stable and using \eqref{coer}, we conclude that
\[
  \int_0^{\infty} \|C T(t) x_0\|^2 dt \geq \delta \|x_0\|^2.
\]
This concludes the proof.
\end{proof}
\medskip

In Russell and Weiss
\cite{ruwe92} it is shown that a necessary condition for exact
observability of exponentially stable systems is the following
version of the Hautus test:
\begin{quote}
There exists a constant $m>0$ such that for every $s\in\mathbb C_-$ and every $x \in \mathbf{D}(A)$:
\begin{equation*}
  \|(sI-A)x\|^2 +|{\rm Re}\, s|\,\|Cx\|^2 \ge m|{\rm Re}\, s|^2\|x\|^2,\qquad {\rm (HT)}
\end{equation*}
\end{quote}
Here $\mathbb C_-$ denotes the open left half plane.  The Hautus test
(HT) is sufficient for approximate observability of exponentially
stable systems \cite{ruwe92} and for polynomially stable systems
\cite{jaschn07}. Further, the Hautus test (HT) is sufficient for exact
observability of strongly stable Riesz-spectral systems with
finite-dimensional output spaces \cite{JacZwa01b}, for exponentially
stable systems with $A$ is bounded on $H$ \cite{ruwe92}, and for
exponentially stable systems if the constant $m$ in (HT) equals one
\cite{grca}. Further, the Hautus test (HT) is
sufficient for exponentially stable systems with a normal
$C_0$-group \cite{JaZwSIAM}.
However, in general the Hautus test (HT) is not sufficient for
exponentially stable systems \cite{jazw04b}. We refer the reader to
Russell and Weiss \cite{ruwe92}, and Jacob and Zwart \cite{JacZwa01b,
  jazw04} for more information on this Hautus test.

\section{Problem statement and main results}\label{sec2}

We consider the port-Hamiltonian system in the following form
on a one-dimensional 
spatial domain  
of the form 
\begin{align}
\frac{\partial x}{\partial t}(\zeta,t) &= \left( P_1
  \frac{\partial}{\partial \zeta} + G_0\right) ({\mathcal H}({\zeta})
x(\zeta,t)), \label{eqn:pde0a} \\
x(\zeta,0) &= x_0(\zeta),\label{eqn:pde0b}\\
0&= \begin{bmatrix} W_{B,1} & W_{B,0} \end{bmatrix} \begin{bmatrix} ({\mathcal H}x)(1,t) \\ ({\mathcal
         H}x)(0,t)\end{bmatrix},\label{eqn:pde0c} \\
       y(t) &= \begin{bmatrix} W_{C,1} & W_{C,0} \end{bmatrix} \begin{bmatrix} ({\mathcal H}x)(1,t) \\ ({\mathcal
         H}x)(0,t)\end{bmatrix}, \label{eqn:pde0d} 
\end{align}
where $\zeta \in [0,1]$ and $t\ge 0$, the $n\times n$ Hermitian matrix $P_1$ is  invertible, 
$G_0$ is a  $n\times n$ matrix, $\begin{bmatrix} W_{B,1} & W_{B,0} \end{bmatrix}$ and $\begin{bmatrix} W_{C,1} & W_{C,0} \end{bmatrix}$ 
are full row rank $n\times 
2n$-matrices, and ${\mathcal H}(\zeta)$ is a positive $n\times n$ Hermitian matrix for almost all $\zeta\in [0,1]$  satisfying ${\mathcal H}, {\mathcal  H}^{-1}\in L^{\infty}(0,1;{\mathbb C}^{n\times n})$. Thus, the matrix $P_1\mathcal{H}(\zeta)$ can be diagonalized as $P_1\mathcal{H}(\zeta)=S^{-1}(\zeta)\Delta(\zeta)S(\zeta)$, where 
$\Delta(\zeta)$ is a diagonal matrix and  $S(\zeta)$ is an invertible matrix  for almost all $\zeta\in [0,1]$. We suppose the 
technical assumption that  $S^{-1}$, $S$,  $\Delta: [0,1] \rightarrow \mathbb{C}^{n \times 
n}$ are continuously differentiable.

The function $x$ denotes the state of the system, $u$ the input function and $y$ the output of the system. 
For more information  we refer to \cite{JacZwa12,JacZwa19}. 
%
%

We equip the space $X:=L^2(0,1;\mathbb{C}^n)$ with the 
energy norm $\langle \cdot,\cdot\rangle_X:=\sqrt{\langle \cdot, {\mathcal H} \cdot \rangle}$, 
where $\langle \cdot,\cdot\rangle$ denotes the standard inner 
product on $L^2(0,1;\mathbb{C}^n)$. Note, that  the energy norm is equivalent to the standard norm on  
$L^2(0,1;\mathbb{C}^n)$.

Our standing assumption is that  the operator $A: {\mathbf D}(A)\subset X\to X$ defined by 
\begin{equation}\label{operatorA}
 Ax:= \left( P_1 \frac{d}{d\zeta} + G_0\right) ({\mathcal H}x), \qquad x\in {\mathbf D}(A), 
\end{equation}
\begin{align}
\label{domainA}
{\mathbf D}(A) := &\left\{ x\in X\mid  {\mathcal H}x\in H^{1}(0,1;\mathbb C^n) \right. \\
\nonumber
&\qquad \left. \text{ and } \begin{bmatrix} W_{B,1} & W_{B,0} 
\end{bmatrix} \begin{bmatrix} ({\mathcal H}x)(1) \\ ({\mathcal
         H}x)(0)\end{bmatrix}=0 \right\},
\end{align}
generates a $C_0$-semigroup $(T(t))_{t\ge 0}$ on $X$. This assumption guarantees that the port-Hamiltonian system \eqref{eqn:pde0a}-\eqref{eqn:pde0d} has for every initial condition $x_0\in \mathbf{D}(A)$ a unique classical solution. 

In  \cite{JacKai19} the dual version of the following proposition is shown.

\begin{proposition}\label{ThmJaKa}
If $A$ generates a contraction semigroup and $G_0^*=-G_0$, then the port-Hamiltonian system \eqref{eqn:pde0a}--\eqref{eqn:pde0d} is exactly observable in finite time, that is, there exist a time $t_0>0$ and constant $\delta>0$ such that for every $x_0\in \mathbf{D}(A)$ we have
\[
  \int_0^{t_0} \|y(t)\|^2 dt\ge \delta \|x_0\|^2.
\]
\end{proposition}
\begin{proof}
In \cite[Section 8]{JacZwa19} it is shown that the dual system is again a port-Hamiltonian system. As exact controllability is the dual notion of exact observability, the statement follows from  the results in \cite{JacKai19}.
\end{proof}
\medskip

Our first main result provides a direct proof for strongly stable systems. 
\begin{theorem}\label{thmmain}
Suppose that the matrix
\begin{equation}
\label{eq:15}
  {\mathcal W}_{BC} := \begin{bmatrix} W_{B,1} & W_{B,0} \\ W_{C,1} & W_{C,0}
\end{bmatrix}
\end{equation}
is invertible, the $C_0$-semigroup generated by $A$ is strongly stable and Re$\, G_0=0$. 
Then the port-Hamiltonian system \eqref{eqn:pde0a}--\eqref{eqn:pde0d} is exactly observable in finite time, that is, there exist a time $t_0>0$ and constant $\delta>0$ such that for every $x_0\in \mathbf{D}(A)$ we have
\begin{align}
  \int_0^{t_0} \|y(t)\|^2 dt\ge \delta \|x_0\|^2.
\end{align}
\end{theorem}
\vspace{1ex}

The proof of the main result is given in Section \ref{mainproof}.
In applications, the assumption that the matrix ${\mathcal W}_{BC}$ is invertible is not a restriction as we do not want to measure quantities that we already have set to zero.

The term $G_0+G_0^*$ relates to the internal dissipation of energy, see also Lemma \ref{L4.1}. So an 
interesting question is what happens when Re$\, G_0 :=\frac{1}{2}( G_0+G_0^*) \neq 0$. In the following finite-dimensional example we show that observability can be lost.
\begin{example}
We take as state space ${\mathbb C}^2$ and choose
\[
  A =\begin{bmatrix} -1 & 1 \\ -1 & 0 \end{bmatrix}, \quad C = \begin{bmatrix} \sqrt{2} & 0\end{bmatrix}.
\]
It is easy to see that
\[
  A + A^* = - C^*C
\]
and so $(C,A)$ is observable and $A$ is Hurwitz. However, with 
\[
  G_0 = \begin{bmatrix} -1 & -1 \\ -1 & -1 \end{bmatrix}
\]
we have that $(C,A + G_0)$ is not observable, whereas $A+G_0$ is Hurwitz. \hfill $\Box$
\end{example}
\medskip

Note that if in this example Re$(G_0)$ would be zero, then $(C,A+G_0)$ would be observable. So we see that in general observability can get lost when using general perturbations. However, for port-Hamiltonian system the Hautus test still holds.
\begin{theorem}
\label{thmmain-2}
Suppose that the matrix ${\mathcal W}_{BC}$, see  (\ref{eq:15}), 
is invertible and the $C_0$-semigroup generated by $A$ is exponentially stable. Further, we suppose that the $C_0$-semigroup generated by $A- (\mathrm{Re}\, G_0){\mathcal H}$ is strongly stable.
Then the port-Hamiltonian system \eqref{eqn:pde0a}--\eqref{eqn:pde0d} satisfies the Hautus test (HT), and thus the system is approximately observable.
\end{theorem}

The proof of Theorem \ref{thmmain-2} is given in Section \ref{mainproof-2}.

\section{Proof of Theorem \ref{thmmain}} \label{mainproof}

This section is devoted to the proof of  Theorem  \ref{thmmain} and we denote by $A$ the operator given by \eqref{operatorA}-\eqref{domainA}. Again we  consider  the port-Hamiltonian system \eqref{eqn:pde0a}-\eqref{eqn:pde0d}. Recall that  for every intial condition $x_0\in \mathbf{D}(A)$ the port-Hamiltonian system \eqref{eqn:pde0a}-\eqref{eqn:pde0d} has  a unique classical solution. 
It is useful to define the so called {\em boundary effort}\ and  {\em boundary flow}, which are given by
\begin{align*}
  e_{\partial} &= \frac{1}{\sqrt{2}} \left( ({\mathcal H}x)(1)+( {\mathcal H}x)(0) \right),\\ 
  f_{\partial} &= \frac{1}{\sqrt{2}} \left( P_1( {\mathcal H}x)(1)-P_1( {\mathcal H}x)(0) \right),
\end{align*}
respectively. We write this as a matrix vector product, i.e.,
\begin{equation}
  \label{trafo}
  \left[ \begin{array}{c} f_{\partial} \\ e_{\partial} \end{array} \right] = R_0 \left[ \begin{array}{c} ({\mathcal H} x)(1) \\  ({\mathcal H} x)(0) \end{array} \right], 
\end{equation}
  with $R_0 \in {\mathbb C}^{2n \times 2n}$ defined as
  \begin{equation}
  \label{R_0}
    R_0=\frac{1}{\sqrt{2}}
    \left[\begin{matrix}  P_1 & -P_1 \\  I & I \end{matrix}\right].
  \end{equation}
Note that $R_0$ is invertible. Further, we define
$$ P:=  {\mathcal W}_{BC}^{-*} R_0^{*}  \begin{bmatrix} 0& I \\ I & 0
\end{bmatrix}R_0 {\mathcal W}_{BC}^{-1}.$$

\begin{lemma}
\label{L4.1}
For $x_0\in \mathbf{D}(A)$ and $t\ge 0$ the classical solution of \eqref{operatorA}-\eqref{domainA} satisfies
\begin{align*}
   \langle Ax(.,t), &x(.,t)\rangle_X + \langle x(.,t), Ax(.,t) \rangle_X \\
   = &  \frac{1}{2}\begin{bmatrix}  0 & y^*(t)\end{bmatrix} P  \begin{bmatrix}  0\\ y(t)\end{bmatrix}\\
    & +\mbox{Re}\, \int_0^1   \left({\mathcal H}(\zeta)x(\zeta,t)\right)^* \left( G_0 {\mathcal H}(\zeta) x(\zeta,t)\right) d\zeta.
    \end{align*}
\end{lemma}
\vspace{1ex}
\begin{proof} We calculate
\begin{align*}
    \langle &Ax(.,t), x(.,t)\rangle_X + \langle x(.,t), Ax(.,t) \rangle_X \\
    =\,& \frac{1}{2}\int_0^1 x(\zeta,t)^*  {\mathcal H}(\zeta)\left( P_1 \frac{d}{d\zeta} + G_0\right) ( {\mathcal H}x)(\zeta,t) d\zeta  \\
    \nonumber & +
\frac{1}{2}
    \int_0^1\left(\left( P_1 \frac{d}{d\zeta} + G_0\right) ( {\mathcal H}x)(\zeta,t) \right)^*  {\mathcal H}(\zeta)  x(\zeta,t) \, d\zeta.
  \end{align*}
  Using the fact that $P_1$ is self-adjoint  we get
  \begin{align*}
  \langle &Ax(.,t), x(.,t)\rangle_X + \langle x(.,t), Ax(.,t) \rangle_X\\
    =&
    \frac{1}{2} \int_0^1   \left({\mathcal H}(\zeta)x(\zeta,t)\right)^* \left( P_1 \frac{d}{d \zeta} \left({\mathcal H} x\right)(\zeta,t) \right) d\zeta  \\
   &+
    \frac{1}{2} \int_0^1   \left( \frac{d }{d \zeta} \left({\mathcal
  H} x\right)(\zeta,t)\right)^* P_1 {\mathcal H}(\zeta) x(\zeta,t) d\zeta  \\
   &+
    \frac{1}{2} \int_0^1   \left({\mathcal H}(\zeta)x(\zeta,t)\right)^* G_0 {\mathcal H}(\zeta) x(\zeta,t) d\zeta\\
     &+
    \frac{1}{2} \int_0^1  \left( {\mathcal H}(\zeta) x(\zeta,t)\right)^* G_0^* {\mathcal H}(\zeta) x(\zeta,t)  d\zeta\\
   =&  \frac{1}{2}\int_0^1 \frac{d }{d \zeta}\left(  \left({\mathcal H}  x\right)^* (\zeta,t) P_1  \left({\mathcal H} x\right)(\zeta,t) \right) d\zeta\\
    & +\mbox{Re}\, \int_0^1   \left({\mathcal H}(\zeta)x(\zeta,t)\right)^* \left( G_0 {\mathcal H}(\zeta) x(\zeta,t)\right) d\zeta\\
    =&  \frac{1}{2} \left(  \left({\mathcal H}
  x\right)^* (1,t) P_1  \left({\mathcal H} x\right)(1,t)- \left({\mathcal H}
  x\right)^* (0,t) P_1  \left({\mathcal H} x\right)(0,t) \right)\\
  & +\mbox{Re}\, \int_0^1   \left({\mathcal H}(\zeta)x(\zeta,t)\right)^* \left( G_0 {\mathcal H}(\zeta) x(\zeta,t)\right) d\zeta.
  \end{align*}
  Combining this equality with (\ref{trafo}) and the definitions of $R_0$ and $P$, we obtain the statement of the lemma.
\end{proof}

\begin{proposition}\label{propLya}
For $x_0\in \mathbf{D}(A)$ we have 
\begin{align*}
   \langle Ax_0, &x_0\rangle_X + \langle x_0, Ax_0 \rangle_X \\
   \ge  & - \|P\|   \|C x_0\|^2\\
    & +\mbox{Re}\, \int_0^1   \left({\mathcal H}(\zeta)x_0(\zeta)\right)^* \left( G_0 {\mathcal H}(\zeta) x_0(\zeta)\right) d\zeta,
    \end{align*}
    where $C x_0:=\begin{bmatrix} W_{C,1} & W_{C,0} \end{bmatrix} \begin{bmatrix} ({\mathcal H}x_0)(1) \\ ({\mathcal
         H}x_0)(0)\end{bmatrix}$.
\end{proposition}
\vspace{1ex}

\begin{proof} \emph{of Theorem \ref{thmmain}: }Follows directly from Proposition \ref{propLya} together with  Proposition \ref{Lyap0}.
\end{proof}

\section{Proof of Theorem \ref{thmmain-2}} \label{mainproof-2}

We start with a general lemma. 
\begin{lemma}
\label{HT+G0}
 Consider the system (\ref{eqnobs1})--(\ref{eqnobs2}) with $A$ the infinitesimal generator of a strongly stable semigroup and assume that this system is exactly observable in finite time. Let $G$ be a bounded operator. Then  there exists a $m >0$ and $\alpha>0$ such that for all \textrm{Re}$(s) <-\alpha$ and $x \in \mathbf{D}(A)$ there holds 
 \[
   \|(sI-A-G)x\|^2 + | \mathrm{Re}(s) \|Cx\|^2 \geq m \mathrm{Re}(s)^2 \|x\|^2.
\]
\end{lemma}

\vspace{1ex}
So we can interpret this as that the Hautus test holds on a smaller left half plane. 

\begin{proof}
By Proposition \ref{Lyap0} we know that that there exists a self-adjoint bounded operator $L$ such that $L \geq \delta I$ and (\ref{lyap-eq}) holds. As $L \geq \delta I$, $\langle \cdot, L \cdot\rangle$ defines a norm on $X$ and so for $x\in \mathbf{D}(A)$ 
\begin{align}
 \label{eq:20}
  &\langle (sI-A)x, L(sI-A)x \rangle \leq 2 \langle Gx, LGx \rangle + \\
\nonumber 
  &\qquad  2 \langle (sI-A-G)x, L(sI-A-G)x \rangle.
\end{align}
Using the Lyapunov equation we have for $x\in \mathbf{D}(A)$ and $s=r+i\omega$ that
\begin{align}
\nonumber
 &\langle (sI-A)x, L(sI-A)x \rangle  \\
 \nonumber
 & \quad =  \langle ((r+ i\omega)I-A)x, L((r+i\omega) I-A)x \rangle \\
 \nonumber
 &\quad =  r^2 \langle x, Lx \rangle + \\
 \nonumber 
 & \quad\quad\quad r \langle x, L(i\omega I-A)x \rangle +  r \langle (i\omega I-A)x, Lx \rangle  +\\
 \nonumber
 & \quad\quad\quad  \langle (i\omega I-A)x, L(i\omega I-A)x \rangle\\
 \nonumber
 &\quad = r^2 \langle x, Lx \rangle + r \|Cx\|^2 + \langle (i\omega I-A)x, L(i\omega I-A)x \rangle\\
 \label{eq:21}
 &\quad\geq r^2 \langle x, Lx \rangle + r \|Cx\|^2. 
\end{align}
So from (\ref{eq:20}) and (\ref{eq:21}) we find that
\begin{align*}
  &2 \langle (sI-A-G)x, L(sI-A-G)x \rangle \\
  & \geq \langle (sI-A)x, L(sI-A)x  - 2 \langle Gx, LGx \rangle\\
  & \geq r^2 \langle x, Lx \rangle + r \|Cx\|^2 - 2 \langle Gx, LGx \rangle\\
  & \geq r^2 \delta \|x\|^2 - 2 \|L\|\|G\|^2 \|x\|^2 + r \|Cx\|^2 \\
  & \geq r^2 \frac{\delta}{2} \|x\|^2 + r \|Cx\|^2 ,
\end{align*}
where the last inequality for $r < -\sqrt{\frac{2\|L\|}{\delta}}\|G\|$. This implies the statement of the lemma.
\end{proof}
\vspace{1ex}

\begin{lemma}\label{lsgode}
Let $s\in {\mathbb C}$ and $P_1,G_0$ and ${\mathcal H}$ as in Section \ref{sec2}. Then the  solution of the system of ordinary differential equations
\begin{equation}\label{ode}
   sx(\zeta)= \left(P_1\frac{d}{d \zeta} + G_0 \right)({\mathcal H} x)(\zeta) + f(\zeta), \quad \zeta \in [0,1],
\end{equation}
is given by 
\begin{align*}
x(\zeta)=\Psi^s(\zeta,0)x(0) +\int_0^{\zeta} \Psi^s(\zeta,\tau){\mathcal H}^{-1}(\tau)P_1^{-1} f(\tau) d\tau,
\end{align*}
where $\Psi^s(\zeta,\tau)$ satisfy for $v\in {\mathbb C}^n$, and $\zeta,\tau \in [0,1]$
\begin{align*}
    \tilde M e^{-|\mathrm{Re}(s)| \tilde c_0(\zeta-\tau)}&\| v\|\\
     \leq &\| \Psi^s(\zeta,\tau)v \| \leq M e^{|\mathrm{Re}( 
s)| c_0(\zeta-\tau)}\|v\|,\end{align*}
with constants $M,\tilde M>0$, and $\tilde c_0, c_0\geq 0$ independent of $s$, $v$, and $\zeta$.
\end{lemma}
\begin{proof}
Equation \eqref{ode} is a linear non-homogeneous ordinary differential equation in ${\mathcal H}x$  with non-homogenity ${\mathcal H}^{-1}P_1^{-1}f$.
This shows the representation of the solution. It remains to prove the estimate of $\Psi^s(.,.)$. Thus in the following we study the corresponding 
homogeneous ordinary differential equation 
\begin{equation}\label{odehom}
   sx(\zeta)= \left(P_1\frac{d}{d \zeta} + G_0 \right)({\mathcal H} x)(\zeta), \quad \zeta \in [0,1].
\end{equation}
Writing $\tilde x={\mathcal H} x$, 
\eqref{odehom} can be equivalently written as 
 \begin{align*}
  {\mathcal H}(\zeta) P_1 \tilde x'(\zeta)=s\tilde x(\zeta)-{\mathcal H}(\zeta) G_0\tilde x(\zeta), 
 \end{align*}
where the prime denotes the (spatial) derivative.

We write $s=i\omega+r$ with $\omega, \tau \in {\mathbb R}$ and diagonalize $P_1 {\mathcal H}(\zeta)=S^{-1}(\zeta)\Delta(\zeta)S(\zeta)$, see Section \ref{sec2}. Thus,
$
{\mathcal H}(\zeta) P_1
=S^{*}(\zeta)\Delta(\zeta)S^{-*}(\zeta)
$
and we get
\begin{align*}
 \Delta&(\zeta) S^{-*}(\zeta) \tilde x'(\zeta)=
  i\omega S^{-*}(\zeta) \tilde 
x(\zeta)+ \\
  &\left(r I-S^{-*}(\zeta) {\mathcal H}(\zeta) G_0 
S^{*}(\zeta)\right)S^{-*}(\zeta)\tilde x(\zeta). 
\end{align*}
 Using the substitution 
$z=S^{-*}\tilde x$, 
gives the equivalent differential equation
\begin{align}
 z'(\zeta)&=i\omega \Delta^{-1}(\zeta)z(\zeta)+ (r \Delta^{-1}(\zeta)+Q(\zeta) 
)z(\zeta), 
\label{z}
\end{align}
where $$Q(\zeta):=-\Delta^{-1}(\zeta)S^{-*}(\zeta){\mathcal H}(\zeta) 
G_0 S^{*}(\zeta)+ (S^{-*})'(\zeta)S^{*}(\zeta).$$
Thus, equation \eqref{odehom} is equivalent to equation \eqref{z}.
Due to the fact that $P_1{\mathcal H}(\zeta)$ has real eigenvalues, $\Delta(\zeta)$ is a diagonal, real matrix and
$i\omega\Delta^{-1}(\zeta)$ is a diagonal, purely imaginary matrix. 
We write $\Delta^{-1}(\zeta)=\mathrm{diag}_{k=1,\ldots,n}(\alpha_k(\zeta))$ with $\alpha_k(\zeta):[0,1]\to {\mathbb R}$ and  define $\Phi_{\omega}(\zeta)=\mathrm{diag}(\exp(-i\omega\int_0^{\zeta}\alpha_k(\tau)d\tau))$ which satisfies
\[
 \|\Phi_{\omega}(\zeta)\|_{{\mathcal L}({\mathbb C}^n)} 
 =1, \qquad \zeta\in[0,1].
\]

Multiplying \eqref{z} with $\Phi_{\omega}(\zeta)$, we get
\begin{align*}
\Phi_{\omega}(\zeta) & z'(\zeta)-i\omega 
\Delta^{-1}(\zeta)\Phi_{\omega}(\zeta)z(\zeta)=\\
& \Phi_{\omega}(\zeta)\left(r \Delta^{-1}(\zeta)+Q(\zeta) \right)z(\zeta) 
\end{align*}
or equivalently 
\begin{align*}
(\Phi_{\omega}&(\zeta)z(\zeta))'\\=
 &(r\Delta^{-1}(\zeta)+\Phi_{\omega}(\zeta)
Q(\zeta)\Phi_{\omega}^{-1}(\zeta))\Phi_{\omega}(\zeta)z(\zeta). 
\end{align*}
Using the substitution $y=\Phi_{\omega}z$, this ordinary differential equation becomes
\begin{equation}\label{yode}
  y'(\zeta)=(r \Delta^{-1}(\zeta)+ Q_{\omega}(\zeta) ) y(\zeta), 
\end{equation}
where $ Q_{\omega}(\zeta):=\Phi_{\omega}(\zeta)Q(\zeta)\Phi_{\omega}^{-1}(\zeta)$. 

There exist constants $c_0, c_1 \geq 0$,  independent of $\omega$, such that
\begin{equation}\label{gr2}
   2 \max_{\zeta \in [0,1]}\|r\Delta^{-1}(\zeta)+ Q_{\omega}(\zeta)\| \leq |r| c_0+c_1.
\end{equation}
The solution $y$ of \eqref{yode} satisfies
\begin{align}\label{eqn:ode1}
\frac{d}{d\zeta}\|y(\zeta)\|^2
&= y(\zeta)^*[(r \Delta^{-1}(\zeta)+ Q_{\omega}(\zeta) )y(\zeta)] + \\
&\quad [(r \Delta^{-1}(\zeta)+ Q_{\omega}(\zeta) )y(\zeta)]^* y(\zeta).\nonumber
\end{align}
This together with \eqref{gr2} implies
\begin{equation*}
-(|r| c_0+c_1)\|y(\zeta)\|^2\le \frac{d}{d\zeta}\|y(\zeta)\|^2\le (|r| c_0+c_1)\|y(\zeta)\|^2,
\end{equation*}
and thus 
\begin{align*}\label{gronwall}
e^{- (|r| c_0+c_1)(\zeta-\tau)}&\|y(\tau)\|^2\\
& \le \|y(\zeta)\|^2\leq 
e^{ (|r| c_0+c_1)(\zeta-\tau)} \|y(\tau)\|^2.
\end{align*}
As the mapping $x\mapsto y$ is boundedly invertible on $
 L^2(0,1,\mathbb C^n)$, with norm independent on $\omega$, the statement follows.
\end{proof}
\medskip

We are now in the position to prove Theorem \ref{thmmain-2}.

\begin{proof} {\em of Theorem \ref{thmmain-2}}:
By Theorem \ref{thmmain} and Lemma \ref{HT+G0} the Hautus Test (HT) is satisfied on some left half plane $\{s\in \mathbb C\mid \mathrm{Re}s<\alpha \}$. Thus,
if the Hautus test (HT) is not satisfied, then there would exist a sequence of complex numbers $\{s_n\}_{n \in {\mathbb N}} \subset {\mathbb C}^-$ with $\sup_n|\mathrm{Re}\,s_n|<\infty $ and a sequence of elements in the domain of $A$,  $\{x_n\}_{n \in {\mathbb N}}$, with norm one such that
\begin{equation}
\label{eq:28}
  \lim_{n \rightarrow \infty} \|(s_n I - A)x_n \|^2 + |\mathrm{Re} (s_n)| \|Cx_n\|^2 = 0.
\end{equation}
Since $A$ generates an exponentially stable semigroup, the real part of $s_n$ cannot converges to zero.
We write $s_n = r_n +  i\omega_n$, and define $f_n= (s_n I - A)x_n $. Using the definition of $A$, we see that $(s_n I - A)x_n =f_n$ can be rewritten as the ode
\begin{equation}
\label{eq:23}
  s_n x_n(\zeta)= \left(P_1\frac{d}{d \zeta} + G_0 \right)({\mathcal H} x_n)(\zeta) + f_n(\zeta) .
\end{equation}

By Lemma \ref{lsgode} we know that the solution of (\ref{eq:23}) is given by
\begin{align*}
  x_n(\zeta)=& \Psi^{s_n}(\zeta,0) x_n(0) +\\
  & \int_0^{\zeta} \Psi^{s_n}(\zeta,\tau) {\mathcal H}^{-1}(\tau)P_1^{-1} f_n(\tau) d\tau.
\end{align*}
Furthermore, since the real parts of $s_n$ are bounded, we find by the same lemma that $\Psi^{s_n}(\cdot,\cdot)$ are uniformly bounded on $[0,1]\times[0,1]$. Moreover, since $r_n$ lies in an interval $(\alpha, \delta)$ with $\delta <0$ we obtain from (\ref{eq:28}) that $Cx_n \rightarrow 0$, and thus  $x_n(0) \rightarrow 0$. By assumption $f_n \rightarrow 0$, so we conclude that $x_n \rightarrow 0$. However, this is in contradiction to the fact that $x_n$ all have norm one.
\end{proof}

\section{Conclusions and remarks}

In the previous section we have shown that port-Hamiltonian systems of the form (\ref{eqn:pde0a})--(\ref{eqn:pde0d}) are exactly observable provided $G_0^*+G_0=0$, and satisfy the Hautus test for any bounded $G_0$. Since the Hautus test implies approximate observability we have that this weaker form of observability does hold. Whether or not port-Hamiltonian systems of the form (\ref{eqn:pde0a})--(\ref{eqn:pde0d}) are always exactly observable is an open problem. We assert that the answer to that question is yes.

Since controllability and observability are dual properties, similar controllability results hold if we have $n$-dimensional control at the boundary. Thus the (stable) system (\ref{eqn:pde0a})--(\ref{eqn:pde0b}) with control 
\[
  u(t) = \left[ \begin{array}{cc} W_{B,1} & W_{B,0}\end{array} \right]\left[ \begin{array}{c} ({\mathcal H}x)(1,t) \\ ({\mathcal H}x)(0,t) \end{array} \right] 
\]
is exactly controllable when $G_0^* + G_0=0$, and satisfies the (dual) Hautus test always.

\noindent
%

%
%
%
%
%
%
%
%
%


\end{document}